\documentclass[11pt, a4paper]{article} 

\title{Ruin probabilities for risk processes \\in a bipartite network}

\author{Anita Behme\thanks{Technische Universit\"at Dresden,
Institut f\"ur Mathematische Stochastik, 01062 Dresden, Germany, email: anita.behme@tu-dresden.de}\,\,, Claudia Kl\"uppelberg\thanks{Center for Mathematical Sciences, Technische Universit\"at M\"unchen,  85748 Garching, Boltzmannstrasse 3, Germany, email: cklu@tum.de}\,\,,  Gesine Reinert\thanks{Department of Statistics, University of Oxford, 24-29 St Giles, Oxford OX1 3LB, UK, email: reinert@stats.ox.ac.uk }
}

\usepackage{amssymb}
\usepackage{amsfonts}
\usepackage{amsmath}
\usepackage{amsthm}
\usepackage{comment} 
\usepackage{graphicx}
\usepackage[usenames, dvipsnames]{xcolor}
\usepackage{verbatim}
\usepackage{dsfont}
\usepackage{color}
\usepackage[numbers]{natbib}
\usepackage{relsize}
\usepackage{lmodern}
\usepackage{url}
\usepackage{MnSymbol}
\usepackage[english]{babel}
 \usepackage{dsfont}  
\usepackage{tikz}
\usetikzlibrary{arrows}
\usepackage{pgfplots}
\usepackage[titletoc]{appendix}

\numberwithin{equation}{section}

\usepackage{caption} 
\newtheorem{theorem}{Theorem}[section]
\newtheorem{lemma}[theorem]{Lemma}
\newtheorem{remark}[theorem]{Remark}
\newtheorem{remarks}[theorem]{Remarks}
\newtheorem{example}[theorem]{Example}
\newtheorem{proposition}[theorem]{Proposition}
\newtheorem{definition}[theorem]{Definition}

\newtheorem{corollary}[theorem]{Corollary}
\newtheorem{fig}[theorem]{Figure}

\newcommand{\bthe}{\begin{theorem}}
\newcommand{\ethe}{\end{theorem}}

\newcommand{\ben}{\begin{enumerate}}
\newcommand{\een}{\end{enumerate}}

\newcommand{\bit}{\begin{itemize}}
\newcommand{\eit}{\end{itemize}}

\newcommand{\beq}{\begin{equation}}
\newcommand{\eeq}{\end{equation}}

\newcommand{\ble}{\begin{lemma}}
\newcommand{\ele}{\end{lemma}}

\newcommand{\bde}{\begin{definition}\rm}
\newcommand{\ede}{\halmos\end{definition}}

\newcommand{\bco}{\begin{corollary}}
\newcommand{\eco}{\end{corollary}}

\newcommand{\bpr}{\begin{proposition}}
\newcommand{\epr}{\end{proposition}}

\newcommand{\brem}{\begin{remark}\rm}
\newcommand{\erem}{\halmos\end{remark}}

\newcommand{\bproof}{\begin{proof}}
\newcommand{\eproof}{\end{proof}}

\newcommand{\bexam}{\begin{example}\rm}
\newcommand{\eexam}{\halmos\end{example}}

\newcommand{\bfi}{\begin{fig}}
\newcommand{\efi}{\end{fig}}

\newcommand{\btab}{\begin{tab}}
\newcommand{\etab}{\end{tab}}

\newcommand{\beao}{\begin{eqnarray*}}
\newcommand{\eeao}{\end{eqnarray*}\noindent}

\newcommand{\beam}{\begin{eqnarray}}
\newcommand{\eeam}{\end{eqnarray}\noindent}

\newcommand{\barr}{\begin{array}}
\newcommand{\earr}{\end{array}}

\newcommand{\bdis}{\begin{displaymath}}
\newcommand{\edis}{\end{displaymath}\noindent}

\def\N{{\mathbb N}}

\def\P{{\mathbb P}}
\def\E{{\mathbb E}}
\def\R{{\mathbb R}}

\def\P{\mathbb{P}}

\def\A{\mathbb{A}}

\def\cals_+{{\cals_+}}
\def\cala{{\mathcal{A}}}

\def\cals{{\mathcal{S}}}

\def\calo{{\mathcal{O}}}

\def\bone{{\mathds 1}}

\newcommand{\la}{{\lambda}}

\newcommand{\var}{{\rm Var}}

\newcommand{\PK}{Pollaczek-Khintchine}

\newcommand{\ov}{\overline}

\newcommand{\Mid}{\,\Big|\,}

\newcommand{\halmos}{\quad\hfill\mbox{$\Box$}}

\def\P{{\bf {\mathbb{P}}}}

\newcommand{\bfa}{\mathbf{a}}

\allowdisplaybreaks[4]

\begin{document}

%\date{}

\maketitle

\begin{abstract}
This paper studies risk balancing features in an insurance market by evaluating ruin probabilities for single and multiple components of a multivariate compound Poisson risk process.
The dependence of the components of the process is induced by a random bipartite network.
In analogy with the non-network scenario, a network ruin parameter is introduced. This random parameter, which depends on the bipartite network, is crucial for the ruin probabilities. Under certain conditions on the network and for  light-tailed claim size distributions we obtain Lundberg bounds and, for exponential claim size distributions, exact results for the ruin probabilities. For large sparse networks, the network ruin parameter is approximated by a function of independent Poisson variables. 
\end{abstract}



%%%%%%%%%%%%%%%%%%%%%%%%%%%%%%%%%%%%%%%%%%%%%%%%%%%%%%%%%%%%%%%%%%%%%%%%%%%
\section{Introduction}\label{s1}
%%%%%%%%%%%%%%%%%%%%%%%%%%%%%%%%%%%%%%%%%%%%%%%%%%%%%%%%%%%%%%%%%%%%%%%%%%%%

Consider an {{\em insurance risk process} in the celebrated Cram\'er-Lundberg model with Poisson claim arrivals, premium rate $c$ and claim sizes $X_k$}, that is, a spectrally positive compound Poisson process $R=(R(t))_{t\ge0}$ given by
$$R(t) = \sum_{k=1}^{N(t)} X_k -ct,\quad t\ge 0,$$
where $R(0)=0$, $c>0$ is a constant, $X_k>0$, $k\in \N$, are i.i.d. random variables with distribution $F$ and finite mean $\mu=\E[X_k]$, and $(N(t))_{t\ge0}$ is a Poisson process with intensity $\lambda>0$. For such a process, the ruin probability for a given risk reserve $u> 0$ is denoted by
$\Psi(u) = \P(R(t) \geq u \text{ for some } t>0)$
and it is given by the famous Pollaczek-Khintchine formula (cf. \cite[VIII (5.5)]{As}, \cite[IV (2.2)]{AsAl} or \cite[Eq. (1.10)]{EKM1997})
\begin{eqnarray}\label{PK}
\Psi(u) = 1-(1-\rho)\sum_{n=0}^\infty \rho^n F_I^{n*}(u)=  (1-\rho)\sum_{n=1}^\infty \rho^n  \ov{ F_I^{n\ast}}(u), \quad u \ge 0,
\end{eqnarray}
whenever the \emph{ruin parameter} $\rho =\la\mu/c$ satisfies $\rho<1$.  
 Hereby, for every distribution function $G$ with $G(0)=0$, for $x\ge0$, we denote the corresponding tail by $\ov G(x) = 1-G(x)$, the {\em integrated tail distribution function} by $G_I(x) = \frac1{\nu}\int_0^x \ov G(y) dy$ if the mean $\nu=\int_0^\infty xdG(x)$ is finite, and the $n$-fold convolution by 
  $G^{n\ast}$, where $G^{0\ast} (x):= \bone\{x\ge 0\}$ and  $G^{(n+1)\ast} (x)  = \int_0^x G^{n\ast}(x-u) dG(u)$, $n \ge 0$.
  
We also recall that, whenever $\rho\ge 1$, then $\Psi(u)=1$ for all $u>0$. Note that the function $\Psi$ of the Pollaczek-Khintchine formula \eqref{PK} is a compound geometric distribution tail with parameter $\rho$. Thus the smaller $\rho$, the smaller the ruin probability, 
and for $\rho<1$ the ruin probability $\Psi(u)$ tends to $0$ as $u\to\infty$.  More precisely, it is well known that, when the distribution function $F$ is \emph{light-tailed} in the sense that an {\em adjustment coefficient} $\kappa$ exists; i.e.,
\begin{equation}\label{lucoeff} 
\exists \kappa>0: \quad  \int_0^\infty e^{\kappa z} dF_I(z) = \frac1{\rho},
\end{equation}
then the ruin probability $\Psi(u)$ satisfies the famous Cram\'er-Lundberg inequality (cf. \cite[Eq. XIII (5.2)]{As}, \cite[Eq. I.(4.7)]{AsAl} or \cite[Eq. (1.14)]{EKM1997})
\begin{equation}\label{classicLundbergbound} 
\Psi(u)\leq e^{-\kappa u} \quad \text{for all}\quad u> 0.
\end{equation}

It is easy to see that $\rho<1$ is a necessary condition for the existence of an adjustment coefficient. 
Further, if the $X_k$ are exponentially distributed with mean $\mu$, then $\kappa=(1-\rho)/\mu$ and
\begin{equation}\label{exporuin}
\Psi(u) = \rho e^{-u(1-\rho)/\mu} \quad \text{for all}\quad u> 0.
\end{equation}
In this paper we derive multivariate analogues to the above classic results in a network setting.
More precisely we consider a multivariate compound Poisson process whose dependency structure stems from a random {bipartite} network which is described in detail in Section~\ref{s2} below. 
We investigate the influence to the insurance market of sharing exogeneous losses modelled by the network.
Insurance companies or business lines of one insurance company are the agents in the bipartite network of Figure~\ref{fig1},
and the portfolio losses, which are the objects, are shared either by different companies or assigned to different business lines within a company. This can also yield useful scenarios for the risk assessment of risk regulators or scenarios of the competitors of a company, when the underlying selection strategy of the agents is unknown.

Our results assess the effect of a network structure on the ruin probability in a Cram\'er-Lundberg setting.
We show that the dependence in the network structure plays a fundamental role for the ruin probability; i.e., the risk within the reinsurance market or within a company.
The ruin parameter $\rho$, which in itself serves as a risk measure, becomes random and its properties depend on the random bipartite network as well as on the characteristics of the claim amount processes. While the network adjacency matrix describes the random selection process of the agents (given by edge indicators), the weights describe how the agents divide the losses among each other. 

As a prominent network model, we single out the mixed Binomial network model with conditionally independent Bernoulli edge indicators as defined in Section~\ref{s2}. This includes the (deterministic) complete network where all agents are linked to all objects and vice versa as a special case. We also provide scenarios for the division of the losses. In the network with homogeneous weights,
 every claim size is equally shared by all agents that are connected to it. The exponential system uses weights which depend on the expected object losses. Notably, for exponentially distributed object claims the 
 exponential system yields an explicit formula for the ruin probability in the network.

Our framework is related to the two-dimensional setting in \cite{AvramPalmPist2,  AvramPalmPist1} where it is assumed that two companies divide claims among each other in some prespecified proportions. The main novelty of our setting is that we consider a network of interwoven companies, with emphasis on studying the effects which occur through this random network dependence structure. Our bipartite  network model has already been used in \cite{KKR1, KKR2} to assess quantile-based risk measures for systemic risk.

Our results extend those for multivariate models to a random network situation.
Whereas one-dimensional insurance risk processes have been extensively studied since Cram\'er's introduction in the 1930s, results for multivariate models (beyond bivariate) are scattered in the literature; for a summary of results see \cite[Ch.~XIII(9)]{AsAl}. 
In general dimensions, multivariate ruin is studied e.g. in \cite{Bregman, Ramasubramannian}, where dependency between the risk processes is modeled by a Clayton dependence structure in terms of a L\'evy copula, which allows for scenarios reaching from weak to strong dependence.
Further, in \cite{Collamore1, pan2017exact}, using large deviations methods, multivariate risk processes are treated and so-called ruin regions are studied, that is, sets in $\mathbb{R}^d$ which are hit by the risk process with small probability. 
In contrast, in our setting claims are partitioned and assigned randomly.

The paper is structured as follows. In Section~\ref{s2} we describe the bipartite network model and present two loss sharing schemes that are characterized by homogeneous or proportional weights.
We focus on three ruin situations, namely the ruin of a single agent, the ruin of a risk balanced set of agents, and the joint ruin of all agents.
Section~\ref{s3} derives results for the ruin probabilities of sums of components of the multivariate compound Poisson process with special emphasis on the network influence. 
Here we derive a network Pollaczek-Khintchine formula for component sums and a network Lundberg bound.
In Section~4 we present explicit results for an exponential system.
In Section~5 we investigate the bipartite network with  conditionally independent edges, and provide a Poisson approximation for the ruin parameter $P$ for arbitrary sets of agents. We specialize such networks to the mixed Binomial network and the complete network.
Section~\ref{s4} is dedicated to the joint ruin probability of all agents in a selected group and provides a network Lundberg bound for this ruin event. The proofs are found in Section~\ref{proof}.

%%%%%%%%%%%%%%%%%%%%%%%%%%%%%%%%%%%%%%%%%%%%%%%%%%%%%%%%%%%%%%%%%%%%%%%%%%%
\section{The bipartite network model}\label{s2}
%%%%%%%%%%%%%%%%%%%%%%%%%%%%%%%%%%%%%%%%%%%%%%%%%%%%%%%%%%%%%%%%%%%%%%%%%%%%

Let $V=(V_1,\ldots,V_d)^\top$ be a $d$-dimensional spectrally positive compound Poisson process with independent components given by
$$V_j(t)=\sum_{k=1}^{N_j(t)}X_j(k) - c_j t, \quad t\ge 0,$$
such that for all $j=1,\ldots, d$ the claim sizes $X_j(k)$ are positive i.i.d. random variables having mean $\mu_j<\infty$ and distribution function $F_j$.  Moreover $N_j=(N_j(t))_{t\ge0}$ is a Poisson process with intensity $\lambda_j>0$, and the premium rate $c_j>0$ is constant. 
The corresponding deterministic constant $\rho_j:=\lambda_j \mu_j/c_j$ as in the Pollaczek-Khintchine formula \eqref{PK} is called {\em ruin parameter of component $j$}.

Further we introduce a random bipartite network, independent of the multivariate compound Poisson process $V$,  that consists of $q$ {\em agents} $\mathcal{A}^i$, $i=1, \ldots, q$, and $d$ {\sl objects} $\mathcal{O}_j$, $j=1, \ldots, d$, and edges between agents and objects as visualized in Figure~\ref{fig1}.

\begin{figure}[t]
	\begin{center}
		\begin{tikzpicture}[-,every node/.style={circle,draw},line width=1pt, node distance=2cm]
		\hspace*{-1cm}
		\node (1)  {$\cala_1$};
		\node (2) [right of=1] {$\cala_2$};
		\node (3) [right of=2] {$\cala_3$};
		\node (4) [right of=3] {$\cala_4$};
		\node (5) [right of=4] {$\cala_5$};
		\node (6) [right of=5] {$\cala_6$};
		\node (8) [below of=2,node distance=2cm] {$\calo_1$};
		\node (9) [right of=8] {$\calo_2$};
		\node (10) [right of=9] {$\calo_3$};
		\node (11) [right of=10] {$\calo_4$};
		\foreach \from/\to in {1/8,1/9,2/8,2/10,3/9,4/8,4/10,4/11,5/10,5/11,6/11}
		\draw (\from) -- (\to);
		\end{tikzpicture}
	\end{center}
	\caption{\label{fig1} A bipartite network with 6 agents and 4 objects. It strongly resembles the depiction of the reinsurance market in Figure~21 of \cite{IAIS}.	}
\end{figure}
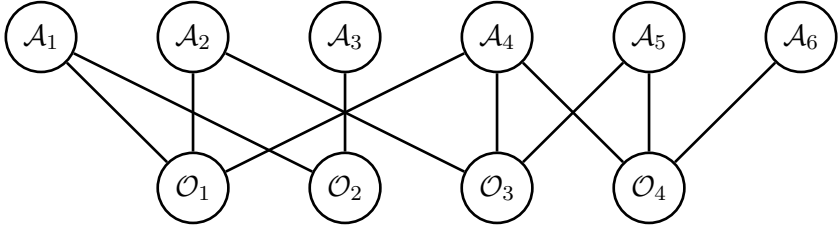

The random edge indicator variables are $\bone \{ i \sim j \}$ indicating whether or not there is an edge between agent $i$ and object $j$. Here and in the following the variable $i$ stands for an agent in $\mathcal{A}$, and the variable $j$ stands for an object in $\mathcal{O}$. 

The weighted edges
 are encoded in a {\em weighted adjacency matrix} 
\beam\label{A}
A=\big(A^i_j\big)_{i=1,\ldots,q \atop j=1,\ldots,d}\quad \text{where} 
\quad A^i_j = \bone \{ i \sim j \} W_j^i
\eeam
for random variables $W_j^i$, which may depend on the random network and have values in $[0,1]$  such that  
\begin{equation}
\label{bedweights}
 0 \le \sum_{i=1}^q A^i_j \le 1  \quad \text{for all} \quad j=1, \ldots, d.
\end{equation}
The edge indicator variables $\bone\{i\sim j\}$ and weights $W^i_j$ may depend on each other but are assumed to be independent of the process $V$.
We use the degree notation
\begin{equation*}\label{sum1}
\deg(i)=\sum_{j=1}^d \bone\{i\sim j \}\quad \text{and}\quad \deg(j)= \sum_{i=1}^q \bone\{i\sim j \}
\end{equation*}
for all $i=1, \ldots, q$, $j=1, \ldots, d$. For $Q\subseteq \{1,\ldots, q\}$ and $j\in\{1,\dots,d\}$ we
abbreviate 
$$\deg(Q)=\sum_{i\in Q} \deg(i)\quad\mbox{and}\quad \mathds{1}\{Q\sim j\}:=\max_{i\in Q}\{\mathds{1}\{i\sim j\}\}.$$
We denote by $\A$  the set of all possible realizations  
$\bfa = (a_j^i)_{i=1,\ldots,q \atop j=1,\ldots,d}$
of the weighted adjacency matrix $A$ from \eqref{A}.

Every object of the bipartite network is assigned to the corresponding component of the compound Poisson process $V$.
Every agent is then assigned to a resulting compound Poisson process, its portfolio, given by
$$R^i(t):=\sum_{j=1}^d A^i_{j}V_j(t),\quad t\ge 0.$$
In total, this yields a $q$-dimensional process $R=  (R^1,\ldots,R^q)^\top$ of all agents given by  
\beam\label{eq-R}
R(t) = A V(t),\quad t\ge 0, 
\eeam
with $V= (V_1,\ldots,V_d)^\top$ as defined above. 
Hence the components of $R$ are no longer independent.

\begin{remarks}\rm
	(i) The independence assumption on the components of $V$ entails that claims in different components never happen at the same time. This is no  mathematical restriction of the model, since we can disentangle dependence through the introduction of additional objects. For example, we can always write two dependent compound Poisson processes $V_1,V_2$ as $V_1=W_1+W_3$ and $V_2=W_2+W_3$, where $W_1$ and $W_2$ have claims only in $V_1$ and $V_2$, respectively, and $W_3$ is the process of the joint claims. Then $W_1, W_2, W_3$ are independent. Thus, mathematically, a third object, 3, is introduced, and objects 1 and 2 are altered. There is a caveat in that this procedure introduces preconditions on the network structure: The resulting edge indicator variables to the new objects 1 and 3 will not be independent as $\bone\{i\sim 1\}=1$ implies $\bone\{i\sim 3\}=1$ for any $i$, and the same holds for edges to objects 2 and 3.  \\[2mm]
	(ii) We can easily extend this model to multiple layers, where e.g. the agents are connected to a set of super-agents via another bipartite network that is encoded in a second weighted adjacency matrix $B$. 
	The resulting process on the top layer is simply obtained by matrix multiplication in \eqref{eq-R}, resulting in $R=BAV$, which reduces the problem to the form \eqref{eq-R}.
\end{remarks}

While many general results in this paper do not require independence of the edge indicator variables in the bipartite network, some of our examples will assume that the edges are conditionally independent, given the value
	of a random variable $\Theta$ which is assumed for convenience to take values in $[0,1]$. One could
	think of $\Theta$ as a hidden variable such as an economic indicator or an environmental variable which governs the behaviour of all agents. Given a realisation $\Theta=\theta$ we then use the notation $p_{i,j}(\theta):= \P ( i \sim j):=\P(\bone\{i\sim j\}=1) \in [0,1]$. The following random bipartite network is of particular interest: 
\begin{itemize} 
	\item The \emph{mixed Binomial network, where $\bone\{i\sim j\}$ are conditionally independent Bernoulli random variables with random parameter $\Theta\in [0,1]$. In case of a degenerate variable $\Theta=p$ a.s., we call the resulting model a \emph{Bernoulli network}, where $\bone\{i\sim j\}$ are independent Bernoulli random variables with parameter $p$. For $\Theta=1$ a.s. we obtain
	the \emph{complete network}, where $\bone\{i\sim j\}\equiv 1$, that is, all agents are linked to all objects and vice versa. }
\end{itemize}

We also single out two specific models for the weights of the weighted adjacency matrix \eqref{A}, which play a prominent role for the network ruin probability. In both examples, the randomness in $A$ arises solely from the randomness of the network; given the network, the weights $W^i_j$ will be deterministic. Still, our general results apply to any random $W^i_j$  as long as  the resulting matrix $A$ is independent of the compound Poisson process $V$ and \eqref{bedweights} holds.
\begin{itemize}
\item
A natural choice for $A_j^i$ is given by the {\em homogeneous weights}
\begin{equation}\label{homo}
A_j^i=\frac{\bone\{i\sim j \}}{\deg(j)}, 
\quad \text{where $\frac{0}{0}$ is interpreted as $0$};
\end{equation}
i.e., every object is equally shared by all agents that are connected to it.
   \item 
A leading example in our paper (see Section \ref{s31}) extends the one-dimensional  precise ruin probability  \eqref{exporuin} for exponentially distributed claims to the network setting. It relies on \emph{proportional weights} defined as follows. Fix $Q\subseteq\{1,\ldots,q\}$ and set for every agent $i\in Q$
\begin{equation}\label{AlaQ}
W^i_j =  W_j^Q =  \frac{\bone\{Q\sim j\} r^Q}{\sum_{k\in Q}\bone\{k\sim j\} \mu_j}, \quad \text{where $\frac{0}{0}$ is interpreted as $0$},
\end{equation}
with some constant $r^Q>0$.  Here $r^Q$ is chosen such that  
$$\sum_{i=1}^q A_j^i= r^Q \frac{ \bone\{Q\sim j\} }{\mu_j} \frac{\sum_{i=1}^q\bone\{i\sim j\} }{\sum_{i\in Q}\bone\{i\sim j\}} 
\leq 1, \quad \text{for all }j=1,\ldots, d,$$   
and it can be viewed as the proneness of group $Q$ to link to objects.
The resulting random weighted adjacency matrix encodes that the exposure of agent group $Q$ to object $j$ is inversely proportional to the expected claim size of the process associated to that object, while for a fixed object $j$ with mean claim  size $\mu_j$, all $i\in Q$ which link to this object share it in equal proportion.
\end{itemize}

We consider the ruin probability of the sum of a non-empty selected subset $Q\subseteq\{1,\dots,q\}$ of all agents and the probability that these agents face ruin (an \emph{and}-condition), that is 
\begin{align}
\Psi^Q(u)& :=\P\Big(\sum_{i\in Q} (R^i(t)-u^i) \geq 0 \text{ for some } t\geq 0\Big),\\
\Psi^Q_\wedge(u)& :=
\P\Big(\min_{i\in Q} (R^i(t)-u^i) \geq 0   \text{ for some } t\geq 0\Big),\label{jointhit2}
\end{align}
for $u\in[0,\infty)^q$ such that $\sum_{i \in Q}u^i\neq 0$.

If $Q=\{1,\ldots, q\}$ we simply denote $\Psi:=\Psi^Q$, while for $Q=\{i\}$ for $i\in\{1,\ldots,q\}$ we write  $\Psi^i=\Psi^{\{i\}}$. Similarly we write $\Psi^{\{i\}}_\wedge= \Psi^{\{i\}}=\Psi^i$ for $i=1,\ldots,q.$ Note that for every $Q'\subseteq Q\subseteq \{1,\ldots, q\}$,
$$\Psi^{Q}_\wedge\leq \Psi^{Q'}.$$ 

\section{The ruin probability of aggregated risk processes in the network}\label{s3}

We start with $\Psi^Q$ for $Q\subseteq\{1,\dots,q\}$, the ruin probability of a set of agents or the total risk of these agents. 
We will derive two main results for the bipartite network.
First, we generalize the Pollaczek-Khintchine formula of \eqref{PK} and, second, the Lundberg inequality \eqref{classicLundbergbound}.
The proofs rely on the independence of the risk processes and the network and are obtained by conditioning on the network, carefully taking the network properties into account. We postpone them to Section~\ref{proof}.

\subsection{The network Pollaczek-Khinchine formula for component sums}

\begin{theorem}\label{prop_psisum} \rm [Network Pollaczek-Khintchine formula for component sums] \\
For any $Q\subseteq \{1,\ldots, q\}$ the joint ruin probability
\begin{equation*} \label{psisumpartial} \Psi^Q (u) =\P\Big(\sum_{i\in Q} (R^i(t)-u^i) \geq 0 \text{ for some } t\geq 0\Big) \end{equation*} for a given risk reserve $u\in [0,\infty)^q$ such that $\sum_{i \in Q}u^i> 0$ has representation
\begin{align}\Psi^Q (u) & = \P(P^Q<1) \E \Big[ (1- P^Q) \sum_{n=1}^\infty (P^Q)^{n} \ov{(F^Q_I)^{n\ast}}(\sum_{i\in Q} u^i) \Mid  P^Q<1 \Big]  + \P(P^Q\geq 1), \label{ruinsumpartial}
 \end{align}
 where $P^Q$ is defined in \eqref{networkrho} and 
\begin{align}\label{FiQ}
F_I^Q(x, A):= F_I^Q(x)&= \Big(\sum_{j=1}^d \Big(\sum_{i\in Q} A^i_j\Big) \lambda_j \mu_j\Big)^{-1} \sum_{j=1}^d \lambda_j \mathds{1}\{Q\sim j\} \int_0^x \ov{F}_j\Big(\frac{y}{\sum_{i\in Q}  A^i_j}\Big)dy,\quad x\ge0,
\end{align}
is a random integrated tail function depending on the matrix $A$, taking values in the set of cumulative distribution functions on non-negative real numbers.
\halmos
\end{theorem}

In the  network  the random variable, henceforth called the {\em (network) ruin parameter},
\begin{align}
P^Q &:= \frac{\sum_{j=1}^d (\sum_{i \in Q} A^i_j )\la_j \mu_j}{ \sum_{j=1}^d (\sum_{i \in Q} A^i_j ) c_j} \mathds{1}\{\deg(Q) > 0 \}\label{networkrho}\\
&= \sum_{j=1}^d \bone \{Q \sim j\}
\frac{\rho_j}{1+ \sum_{k\neq j}\bone\{Q \sim k\}\frac{\sum_{i\in Q} \bone\{i \sim k\} W^i_k}{ \sum_{i\in Q}\bone\{i \sim j\} W^i_j}\frac{c_k}{c_j}} 
\label{networkrho2}
\end{align}
is the random equivalent of $\rho$ in the classical Pollaczeck-Khinchine formula \eqref{PK}. Note that given $P^Q\geq 1$ as in the classical case we derive from \eqref{ruinsumpartial} that $\Psi^Q(u)=1$. While in the classical case, $\rho < 1$ is a cut-off for $\Psi(u)$ to trivially equal 1,  as $P^Q$ is random, in the network  a similar cut-off for the ruin probability for $P^Q<1$ or $P^Q\geq 1$ is not available.

\begin{remark}\rm \label{deg0notnecessary} If the weighted adjacency matrix $A$ is such that $\deg(Q) =0$, then $\sum_{i\in Q} A^i_j =0$ and $P^Q=\frac{0}{0}:=0$. Hence the indicator $ \mathds{1}\{\deg(Q) > 0 \}$ in \eqref{networkrho} is not mathematically necessary, however, we keep it for transparency. Example~\ref{ex-3.6} illustrates a case where the indicator features prominently. 
\end{remark} 

The following remark collects some general observations on $P^Q$.

\begin{remark}\rm
	(i) Given $\deg(Q)>0$ it holds that
	$$\min \{\rho_j,j=1,\ldots, d\} \leq P^Q\leq \max\{\rho_j,j=1,\ldots, d\}.$$
	Thus, if all objects have a ruin parameter ${\rho_j} <1$, then $P^Q < 1$. Nevertheless $P^Q < 1$ can be achieved even if some ruin parameters exceed 1, as long as the others balance this contribution.  \\[2mm]
	(ii) Eq.~\eqref{networkrho2} shows that the ruin parameter $P^Q$ depends on the weights $W_j^i $ only through ratios of sums of weights. 
	Eq.~\eqref{networkrho2} implies further that 
	$$ P^Q \le \sum_{j=1}^d \bone \{Q \sim j\} { \rho_j}.$$
	This bound is an equality when all agents $i\in Q$ are connected only to one single object. Otherwise the bound may be quite crude. Using the Markov inequality this bound can be used to bound $\P ( P^Q \ge t)$ for any $t >0$:
	$$\P ( P^Q \ge t) \le  \frac{1}{t} \E \Big[ \sum_{j=1}^d \bone \{Q \sim j\} { \rho_j} \Big]  = \frac{1}{t} \sum_{j=1}^d \P(Q\sim j)  \rho_j.$$ 
\end{remark}

\bexam\label{ex-3.6}[Equal ruin parameters]\\
If all $\rho_j=\rho$ are equal, we obtain directly from \eqref{networkrho} that for any set $Q\subseteq\{1,\ldots,q\}$
\begin{align*}
P^Q
&= \rho \mathds{1}\{\deg(Q)> 0\}
\end{align*}
and hence for any measurable function $f$ on $\R$
\begin{align*}
\E [f(P^Q)] 
&= f(\rho) \P(\deg(Q)> 0) + f(0) \P(\deg(Q)= 0).
\end{align*}
In particular,  $\E [P^Q ] <1$ if and only if 
$\rho < (\P(\deg(Q)> 0))^{-1}$. 
Comparing this condition to the condition $\rho < 1$ in the non-network case, we see that the presence of the network allows for $1 \le \rho < (\P(\deg(Q)> 0))^{-1}$. 
The network thus balances the ruin probabilities for single components in the sense that  $\rho>1$ is possible while still ensuring that $\E [P^Q] <1$. Similarly, if we interpret $P^Q$ as a risk measure, then 
	$\sum_{i\in Q} P^{\{i\}} = \rho \sum_{i\in Q} \bone\{\deg(i)>0\}$, which can be much larger than $P^Q=\rho \bone\{\deg(Q)>0\}$. 
\eexam

\bexam[Deterministic weights] \\
Let $W_j^i= \frac{r}{\la_j\mu_j}$ be independent of $i$ and $r> 0$ independent of $i$ and $j$, such that \eqref{bedweights} holds.
Then  
$$  P^Q
=  \sum_{j=1}^d  \sum_{i \in Q} \bone \{i \sim j\}\Big(\rho_j^{-1} \sum_{i \in Q} \bone \{i \sim j\} +\sum_{k\neq j} \rho_k^{-1}\sum_{i \in Q} \bone \{i \sim k\}  \Big)^{-1}. 
$$
The summands in the nominator show that the agents in $Q$ share the ruin parameters of all objects they are linked to in equal proportion.
A small $P^Q$ corresponds to a large denominator, hence, to small $\rho_k$'s. 
Consequently, the agents group $Q$ would favour risk processes with small ruin parameters.
\eexam

For illustration purposes we extract from Theorem \ref{prop_psisum} the ruin probability of a single agent in the network.
	
\bexam \label{thm:ruinprob}{\rm [Network Pollaczek-Khintchine formula for a single agent]}\\
	The ruin probability for a given risk reserve $u^i$ of $R^i$ for $i\in\{1, \ldots, q\}$ is given by
	\begin{align*}
	\Psi^i(u^i)
	&= \P( P^i<1) \E \Big[(1- P^i)  \sum_{n=1}^\infty (P^i)^{n} \ov{(F_I^i)^{n\ast}} (u^i) \Mid P^i<1 \Big]  +   \P(  P^i\geq 1)  , \quad u^i > 0,  
	\end{align*}
	where for $\deg(i)>0$, 
	$$ 	P^i:=P^{\{i\}}= \frac{\sum_{j=1}^d A^i_j \la_j \mu_j}{ \sum_{j=1}^d A^i_j c_j} = \sum_{j=1}^d \bone  \{i \sim j\} \frac{ \rho_j}{1+\sum_{k\neq j} \bone\{i \sim k\} \frac{ W^i_k}{ W^i_j } \frac{c_k}{c_j} }, $$ 
	and 
	\begin{equation*}
	F^i_I(x) := F^{\{i\}}_I(x)
	= \Big(\sum_{j=1}^d A^i_j \la_j \mu_j\Big)^{-1} \sum_{j=1}^d  \la_j \bone\{ A^i_j  \ne 0\} \int_0^x  \ov F_j\Big(\frac{y}{A^i_j}\Big)  dy ,\quad x\ge 0.
	\end{equation*}
\eexam

\begin{remark}\rm \label{remnonetwork}
	In Example~\ref{thm:ruinprob}, if the network is deterministic and fixed, then the formula for the ruin probability of a single agent reduces to the classical Pollaczek-Khintchine formula \eqref{PK}.
	\end{remark}

\subsection{A Lundberg bound for $\Psi^Q $}

As in the classical one-dimensional setting, we expect exponential decay of the ruin probability of sums of agents also in the network setting, provided that the claim size distributions are light-tailed. This is shown in the following theorem.
 Note that a similar result for ruin probabilities of sums of components of a multivariate risk process, but without network structure, is derived in \cite[Ch.~XIII, Proposition 9.3]{AsAl}.\\

In order to find an adjustment coefficient which is independent of the specific realisation of the network, let  $W^i$ be deterministic constants such that for all $j=1, \ldots, d$, 
\begin{equation}\label{wW}
0 \leq W_j^i \le W^i\le 1.  
\end{equation}

\begin{theorem} \label{prop:lundbergbound}{\em [Network Lundberg bound for component sums]}\\
	Let $Q\subseteq\{1,\ldots, q\}$ be a set of agents and assume that for all $j\in \{1,\ldots,d\}$ the cumulant generating functions $\varphi_{j}(t):=\log \E e^{t V_j(1)}$ exist in some neighbourhood of zero. Then for fixed $\bfa\in\A$, 
	\begin{equation} \label{lundbergreal}
	\P\Big(\sum_{i\in Q} (R^i(t)-u^i)\geq 0 \; \text{for some } t\geq 0 \mid A=\bfa\Big) \leq \mathds{1}\{\deg(Q)>0\} e^{-\kappa(\bfa) \sum_{i \in Q}u^i},\, u\in[0,\infty)^q, \sum_{i\in Q} u^i>0,
	\end{equation}
	with 
	$$\kappa(\bfa) =\sup\Big\{r>0: \sum_{j=1}^d \varphi_{j}\Big(r\sum_{i\in Q} a^i_j\Big)\leq 0 \Big\}.$$
	In particular, if for all  $j=1,\dots,d$ an adjustment coefficient $\kappa_j\in(0,\infty)$ satisfying \eqref{lucoeff} exists, then 
	$$\Psi^Q (u) \leq  \P(\deg(Q)> 0) e^{- \kappa \sum_{i \in Q}u^i}, \quad u\in[0,\infty)^q, \sum_{i\in Q} u^i>0,$$
	where 
	\begin{equation}
	\label{kappaglobal}
	\kappa= \frac{\min\{\kappa_1, \ldots, \kappa_d\}}{\sum_{i\in Q} W^i}.\end{equation}
\end{theorem}

Note that the general bound in Theorem~\ref{prop:lundbergbound} is optimal only in the case that all agents in $Q$ are only connected to the objects with the heaviest tail in the claim size distribution. For a given network structure \eqref{lundbergreal} shows that a Lundberg bound can exist even if some of the claim sizes are heavy tailed in the sense that \eqref{lucoeff} does not hold for all objects $j$.

Also note that, similar to Remark \ref{remnonetwork}, for a deterministic and fixed network structure an application of Theorem~\ref{prop:lundbergbound} on a single agent with positive degree yields the classical Lundberg bound \eqref{classicLundbergbound}.

\section{The exponential system} \label{s31} 

For the one-dimensional ruin model the exponential distribution and mixtures thereof are the only claim size models which allow for an explicit solution of \eqref{PK}. Hence, it is not surprising that exponential claim size distributions also play a prominent role in the network model. However, an explicit solution also depends on the network itself. In what follows we work with an {\em exponential system}, which is characterized by the proportional weights as in \eqref{AlaQ}, identical Poisson intensities $\lambda_j=:\lambda$ and exponential claim sizes with means $\mu_j$.
For this exponential system we obtain an explicit expression for \eqref{thm:ruinprob}.

 \begin{theorem} \label{expjumpQ} {\em [Ruin probability for component sums in the exponential system]}\\
Let $Q\subseteq\{1,\dots,q\}$ be a set of agents and assume the exponential system as defined above. 
	Then the ruin probability of the sum of all agents in $Q$ is given by
	\begin{equation}\label{ruinspecialQ}
	\Psi^Q (u) = \P( P^Q<1)  \E \Big[ P^Q  e^{- \frac{ 1 -  P^Q}{r^Q} \sum_{i\in Q} u^i }    \Big |   P^Q<1 \Big] 
	 +   \P( P^Q\geq 1),\quad u\in [0,\infty)^q, \sum_{i\in Q} u^i > 0,
	\end{equation}
	and, regardless of the claim size distribution, 
	\begin{align} \label{networkrhoerlangQ} 
	P^Q  = \la \frac{\sum_{j=1}^d \bone\{Q\sim j\}}{ \sum_{j=1}^d \bone\{Q\sim j\} c_j/\mu_j}.
	\end{align} 
\end{theorem}

In contrast to Theorem \ref{prop_psisum}, in this special case
 the integrated tail distribution from \eqref{FiQ} is deterministic and exponential, 
	$$F^Q_I(x) =  1-e^{-x/r^Q}, \quad x\ge 0.$$

Note that Eq. \eqref{ruinspecialQ} can be  
		abbreviated as
		\begin{equation} \label{ruinasexpoff}
		\Psi^Q(u) = \E [ f(P^Q)] 
		\end{equation}
		with the function $f$ given as 
		$$f(\rho) = \bone \{\rho<1\} \rho  e^{- \sum_{i \in Q} u^i ( 1 -  \rho)/r^Q }  		+  \bone \{\rho \geq 1\}.
		$$

We may again extract the ruin probability for a single agent in the network as follows.

\bexam[Ruin probability for a single agent]\\
	Let $i\in\{1,\dots,q\}$ and assume that the conditions of Theorem~\ref{expjumpQ} hold with $Q=\{i\}$. 
	Then the ruin probability of agent $i$ for $u^i> 0$ is given by
	\begin{equation}\label{ruinforspecialweights}
	\Psi^i(u^i) = \P( P^i\geq 1) + \P( P^i<1) \E \left[ P^i  e^{- u^i( 1 -  P^i)/r^i }    \Big |   P^i<1 \right],
	\end{equation}
	where  
	\begin{align*}
	P^i 
	= \lambda \sum_{j=1}^d  \frac{ \bone \{ i \sim j\}}{\frac{c_j}{\mu_j} + \sum_{k \ne j} \bone \{i \sim k\} \frac{c_k}{\mu_k}}.
	\end{align*} 
	 The argument in the expectation in \eqref{ruinforspecialweights}  coincides with \eqref{exporuin} with random $\rho$.\\	 
For the special situation of equal \PK\ parameters as in Example~\ref{ex-3.6} we obtain the single agent's ruin probability as
\begin{align*}
\Psi^i(u^i)= \left(\bone \{\rho<1\} \rho  e^{- u^i( 1 -  \rho)/r^i } 	+  \bone \{\rho \geq 1\}\right) \P(\deg(i) >  0),\quad u^i\ge 0.
\end{align*}
\eexam

\section{The bipartite network with conditionally independent edges} \label{s42erdos}

Throughout this section we assume that the edge indicators in the bipartite network are conditionally independent, given the random variable $\Theta$, and that for the realisation $\Theta=\theta$ we have $\P(i\sim j)=p_{i,j}(\theta)$.
In this model, for $Q\subseteq\{1,\dots,q\}$, also the degrees for different $i \in Q$ are conditionally independent; 
in particular, 
\begin{equation*}\label{mixedsbm} 
\P (\deg(Q) = 0\mid \Theta=\theta) = \prod_{i \in Q} \prod_{j =1}^d (1 - p_{i,j}(\theta)).
\end{equation*} 
For fixed $\Theta = \theta$ this model is a prominent network model (an inhomogeneous random graph, cf. \cite{Bollobas}),
 and we present results for the network ruin parameter as well as the network ruin probability in several situations.

If $\rho_j=\rho$ for $j=1,\ldots, d$, then from Example~\ref{ex-3.6} we know that the ruin parameter $P^Q = \rho \bone (\deg Q > 0)$.
 In general, calculating $P^Q$ and functions thereof as for example in \eqref{ruinasexpoff} is not easy. 
For sparse networks we therefore give a Poisson approximation for $P^Q$. Here sparseness refers to the sum
of the squared edge probabilities being small.

\subsection{Poisson approximation of $P^Q$}\label{s42}

The results in this subsection are based on the following  proposition, which follows from \cite[Thm. 10.A]{BHJ} by conditioning; the proof is omitted. 

	\begin{proposition} \label{poissoncor} %\label{ex-poisson}
		Assume we are given a bipartite network such that the edge indicators $\bone \{i \sim j\}$ for $i=1, \ldots, q$ and $j=1, \ldots, d$ are conditionally independent given the value of $\Theta \in [0,1]$. 
		For each $\theta \in [0,1]$ let $Z_{i,j}(\theta)\sim {\rm Poisson}(p_{i,j} (\theta))$ be independent Poisson variables  for $i=1, \ldots, q$ and $j=1, \ldots, d$. 
		Then for any $g:{\mathbb{Z}}^{qd} \rightarrow [0,1]$, 
		\begin{align}\label{poissonbound} 
		\big|  \E [ g( \{  \bone \{i \sim j\},  i=1, \ldots, q,  j=1, \ldots, d \})]  - \E [ g(Z_{1,1}(\Theta), \ldots, Z_{q,d}(\Theta))]  \big| & \le {R(\Theta)} \\ &:=\sum_{i=1}^{q} \sum_{j=1}^d   \E [ p_{i,j}(\Theta)^2].\nonumber
		\end{align} 
	\end{proposition}
	
If agents pick objects with probability roughly proportional to the number of objects, so that for some fixed $\alpha > 0$, $p_{i,j} (\theta) \sim \alpha d^{-1}$  for all $\theta$ as $q/d \rightarrow 0$, then the bound \eqref{poissonbound} tends to 0 if $q/d \rightarrow 0$.

\begin{proposition}\label{Poihomo}{\em [Homogeneous weights]}\\
Let $Q\subseteq\{1, \ldots, q\}$ be a set of agents and assume  homogeneous weights as in \eqref{homo}.
For each $\theta \in [0,1]$ set 
		\begin{equation} \label{poissonPigen}
Z^Q (\theta) = 	\sum_{j=1}^d \sum_{i \in Q}\frac{ \lambda_j  \mu_j  Z_{i,j}(\theta)  }{ c_j (1+Z^{(i)}_j(\theta)) + \sum_{\ell \ne j} \sum_{s \in Q } Z_{s, \ell}(\theta)  c_\ell  ( 1 + Z^{(i)} _j(\theta)+ \bar{Z}^{(i)} _j(\theta))  / (1 + Z^{(s)} _\ell(\theta)+ \bar{Z}^{(s)} _\ell(\theta)) },		
		\end{equation}
		where the Poisson variables $Z_j^{(i)}(\theta)$, $\bar{Z}_j^{(i)}(\theta)$ and $Z_{i,j}(\theta)$ are independent with means $\sum_{s\in Q, s \ne i} p_{s,j}(\theta)$, $\sum_{s\not\in Q, s \ne i} p_{s,j}(\theta)$ and $p_{i,j}(\theta)$, respectively. Let ${R(\Theta)} $ be as in  \eqref{poissonbound}. 
		Then for $g: \R \rightarrow [0,1]$, 
		\begin{equation}\label{poissonbound2} 
		| \E [g(  P^Q)] - \E [g( Z^Q(\Theta)]  |  \le \sum_{j=1}^d \rho_j {R(\Theta)}.
		\end{equation} 
\end{proposition} 

\begin{corollary}\label{qualPoi}{\em [Homogeneous weights for a single agent]}\\
If $Q = \{i\}$, then under the assumptions of Proposition  \ref{poissoncor}  the approximating Poisson-based random variable for $P^i$  simplifies to
\begin{equation} \label{poissonPi}
Z^i (\theta) :=
\sum_{j=1}^d \bone \{ i \sim j \} \frac{\rho_j}{1 + \sum_{k \ne j} \bone \{ i \sim k\} \frac{c_k (1 + \bar{Z}_j^{(i)}(\theta)) }{c_j (1 + \bar{Z}_k^{(i)} (\theta)) } },
\end{equation}
where the Poisson variables $\bar{Z}_j^{(i)}(\theta)$ are independent with means 
${\pi_{i,j} (\theta)= \sum_{s \ne i} p_{s,j}} (\theta)$.
\end{corollary}

\begin{remark} 
Based on the Poisson approximation, the Delta method can be used to approximate $\E [P^i] $ using the expressions in  \eqref{poissonPi}. 
For each $\theta \in [0,1]$ set
	$$S_{i,j} (\theta):=
	1 + \sum_{k \ne j} \bone \{ i \sim k\} \frac{c_k (1 + \bar{Z}_j^{(i)}(\theta)) }{c_j (1 + \bar{Z}_k^{(i)}(\theta))},$$
	then  by conditioning on the value of $\Theta$ and of $\bone \{ i \sim j\}$,
$$ \Big| \E [P^i]  - \sum_{j=1}^d \rho_j \E \big[p_{i,j} (\Theta) S_{i,j}^{-1} (\Theta) \big] \Big| \le  \sum_{j=1}^d \rho_j {R(\Theta)}.$$
We calculate 
	\begin{align*} 
	\E [ S_{i, j}{(\theta)} \mid \Theta = \theta] 
	&= 1 + \sum_{k \ne j} p_{i,k} (\theta)
	\left( 1 + \pi_{i,k} {(\theta)}\right) \frac{c_{k} }{c_{j}} \left( \frac{1 - \exp( - \pi_{i,k} (\theta))} {\pi_{i,k} (\theta)} \right)
	=: \beta_{i,j}(\theta),
	\end{align*}
where we used Eq.~(3.9) in \cite{CS} for the last equality. Similarly, by the independence of the Poisson variables and again using the results in \cite{CS},
	\begin{align*}
	\lefteqn{\var [ S_{i,j}(\theta)\mid \Theta = \theta] }\\
	&=&\pi_{i,j}^2 (\theta)\sum_{k \ne j }
	\left( \frac{c_{k} }{c_{j}} \right)^2 p_{i,k}(\theta)
	\left\{ {\rm Chi}(\pi_{i,k}(\theta)) + {\rm Shi} (\pi_{i,k}(\theta)) - \log (\pi_{i,k}(\theta)) -\gamma - \left[ \frac{1 - \exp( -\pi_{i,k}(\theta) )}{ \pi_{i,k} (\theta)} \right]^2\right\},
	\end{align*}
	where ${\rm Chi}(x)$ is the hyperbolic cosine integral, ${\rm Shi}(x)$ is the hyperbolic sine integral, $\log(x)$ is the natural logarithm and $\gamma$ is the Euler-Mascheroni constant.
	If each $S_{i,j}(\theta)$ has small variance $\var [ S_{i,j}(\theta)\mid\Theta = \theta] $, then the Delta method combined with Proposition \ref{poissoncor} yields $\E{P^i}$ can be approximated well by  
	$\sum_{j=1}^d \rho_j \E \left[ \frac{p_{i,j}  (\Theta)}{\beta_{i, j} (\Theta)} \right].$
	If agents pick objects with probability roughly proportional to the number of objects, so that  for some fixed $\alpha > 0$, $p_{i,j} (\theta) \sim \alpha d^{-1}$, then $\var [ S_{i,j}(\theta)\mid \Theta = \theta] \sim (q/d)^2\to 0$ if $q/d \rightarrow 0$.
\end{remark}

\begin{proposition}\label{Poiexpo}{\em [Proportional weights]}\\
Let $Q\subseteq\{1, \ldots, q\}$ be a fixed set of agents.
Assume proportional weights as in \eqref{AlaQ}. 
For each $\theta \in [0,1]$ set 
\beam\label{eq:expoPoi}
 Z^Q(\theta):= \sum_{j=1}^d  Z_{Q,j}(\theta)\frac{ 1
}{\frac{1}{\rho_j} + \sum_{\ell \ne j} Z_{Q,\ell}(\theta) \frac{1}{\rho_\ell} },
\eeam
where $Z_{Q,j}(\theta)$ are independent Poisson variables with means $1 - \prod_{i \in Q} ( 1 - p_{i,j}(\theta))$.
Let ${R(\Theta)}$ be as in  \eqref{poissonbound}. Then for any $g:\R \rightarrow [0,1]$, 
		\begin{equation}\label{poissonbound2}
		 \big|  \E [g(  P^Q)] - \E [g( Z_p^Q(\Theta) ] \big|  \le \sum_{j=1}^d \rho_j {R(\Theta)}.
		\end{equation} 
\end{proposition}

In practical applications the distribution of $\Theta$ may not be available. In such a situation a second approximation step could be used, approximating mixed Poisson variables of the type $Z(\Theta)$ by a Poisson variable $Z$ with mean $\lambda$. By \cite[Thm. 1.C]{BHJ} the approximation error is bounded so that for any $g:\R \rightarrow [0,1]$, we have 
	$
		 \big|  \E [ g(  Z(\Theta) )] - \E [ g( Z) ] \big|  \le  \min(1, \lambda^{-1}) \E [ | \Theta - \lambda|] .
		$ Corresponding error terms  would then be added to the bounds in this subsection.

\subsection{$P^Q$ and $\Psi^Q$ in the mixed Binomial network}

As in the mixed Binomial network all vertices are exchangeable, we can assume without loss of generality that $Q = \{ 1, 2, \ldots, |Q|\}$. 
Further, given $\Theta=\theta \in(0,1]$, every edge is chosen with the same probability $\theta \in(0,1]$ independently, 
the degree $\deg(i)$ of each agent $i \in Q$ follows a Binomial distribution with parameters $d$ and $\theta$. 
{Thus,} for $\Theta=1$ a.s. we obtain the complete network treated in Section~\ref{s42complete}.

In a general mixed Binomial model the value $P^Q$ can take on any positive number. To see this, consider a Bernoulli network with fixed edge probability $p>0$. Then on the set $\{\deg(Q)> 0\}$, in the limit for $p\to 0$ the set of vertices in $Q$ will have exactly one edge, and the corresponding neighbour $J$ of $Q$ is chosen uniformly at random in $\{1,\ldots, d\}$. Hence for $p \to 0$ we approach a single edge network such that 
$$\lim_{p\to 0} \P(P^Q\leq x\mid \deg(Q)> 0)= \P\left(\frac{\lambda_J\mu_J}{c_J}\leq x\right) = \frac{\# \{j: \rho_j \leq x \} }{d},$$ 
where $J$ is uniformly distributed on $\{1,\ldots, d\}$. This is also illustrated in Figure \ref{bernoullimeanvar}, which shows the varying balancing effect of the network on $P^{i}$ when the proportion of objects with high and low ruin parameter is changed.

 \begin{figure}[t] 
	\includegraphics[width=0.5\textwidth]{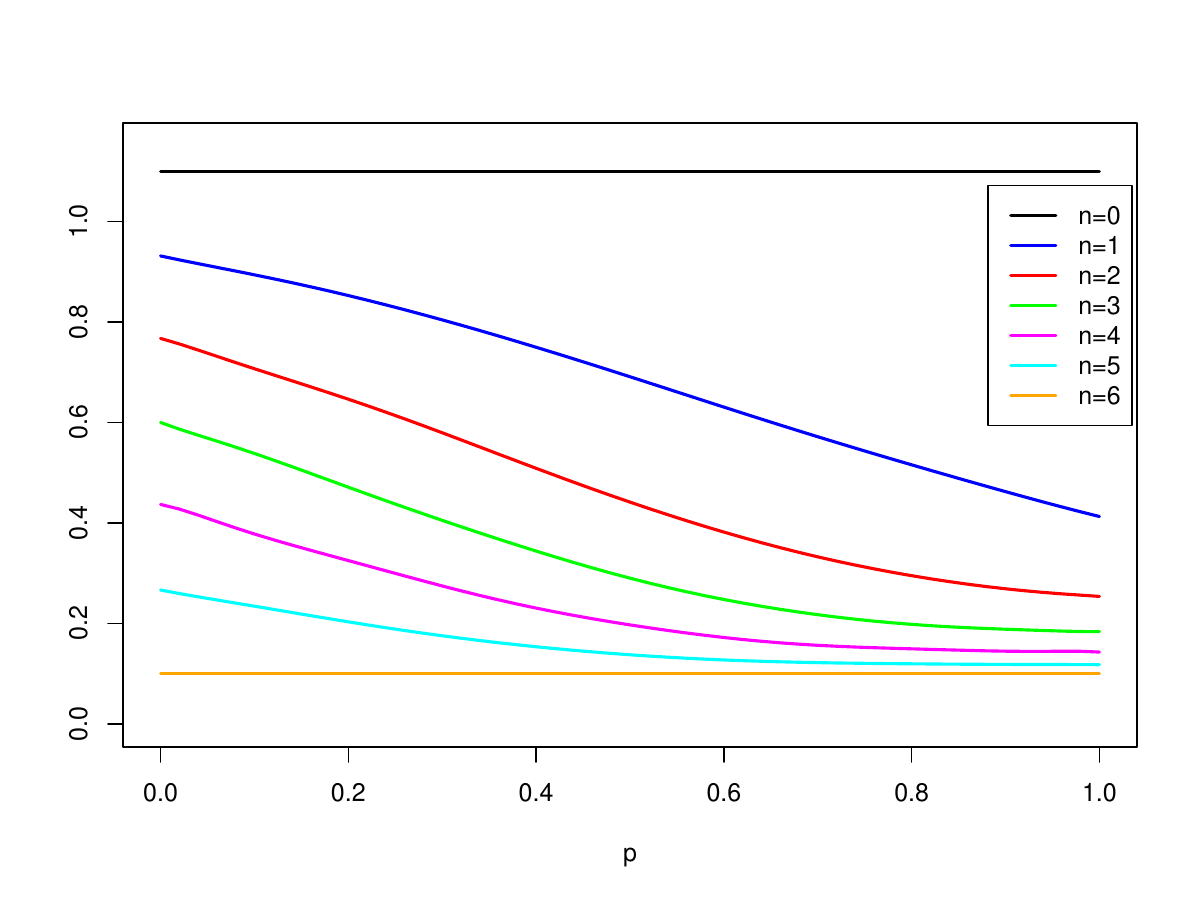}	
	\includegraphics[width=0.5\textwidth]{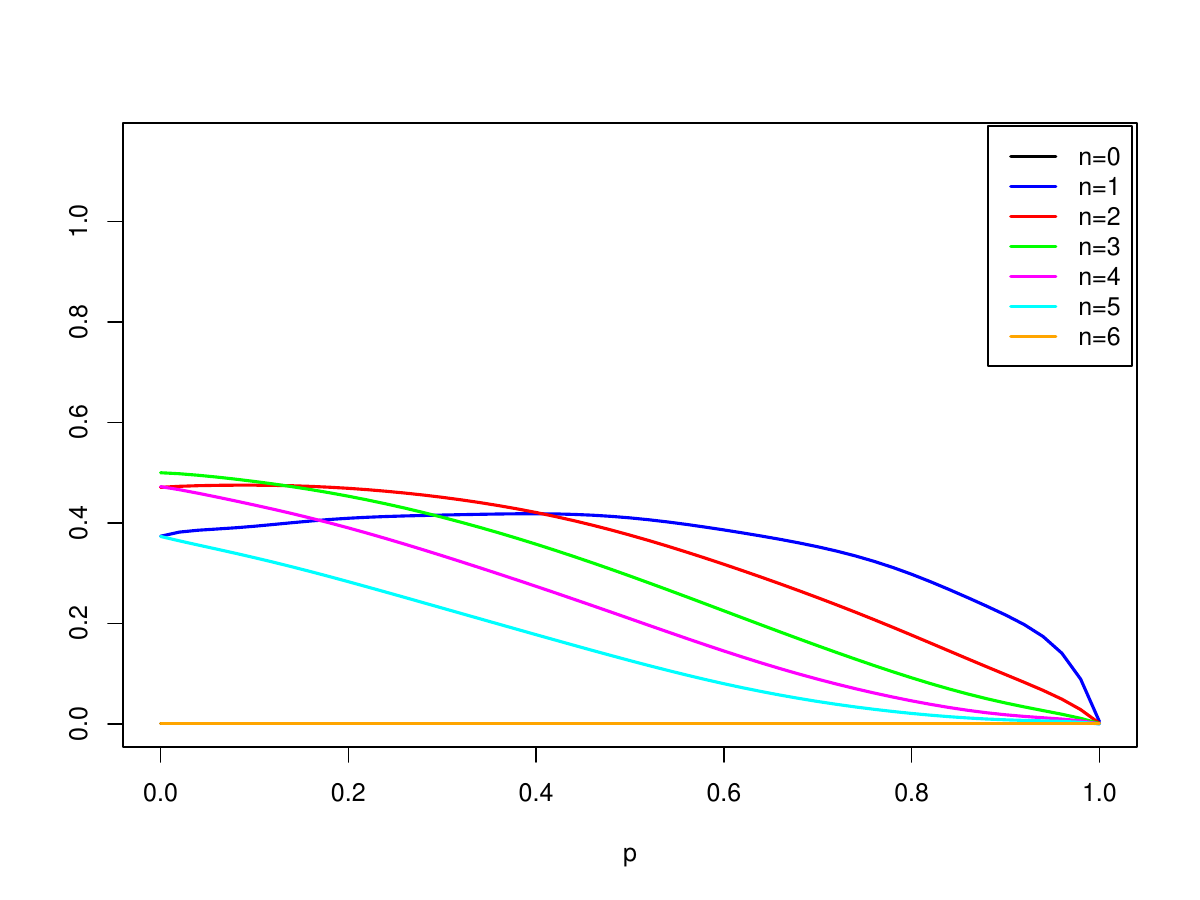}
	\caption{$\E[P^1|\deg(1)> 0]$ (left) and $\sqrt{\var(P^1|\deg(1) > 0)}$ (right) as functions of $p$ in a homogeneous Bernoulli model. Here, $\rho_j=1.1$, $j=1,\ldots,d-n$ and $\rho_j=0.1$, $j=d-n,\ldots,d$; i.e., the proportion of objects with high and low ruin parameter is changed with $n$. Further $\lambda_j\mu_j=0.5$, $j=1,\ldots,d$, and $d=q=6$. 
		Observe that for large $p$ the influence of the proportion on the expected value becomes smaller, because of many connections and a resulting high balancing effect. Still, the behaviour of the standard deviation depends heavily on the proportion as long as there exist two different ruin parameters in the system.}
	\label{bernoullimeanvar} 	
\end{figure}

If in the mixed Binomial model $\rho_j=\rho$ for $j=1,\dots,d$ we know from Example~\ref{ex-3.6} that $P^Q=\rho \bone (\deg(Q)>0)$ 
and as $\P (\deg Q = 0\mid\Theta=\theta) =  (1-\theta)^{|Q|d}$, we have 
 $\E P^Q <1$ if and only if 
 $\rho < \E[(1- (1-\Theta)^{|Q|d})^{-1}]$.
If the distribution of $\Theta$ allows for interchanging expectation and the limit $d \to \infty$, then
 for $d \to \infty$ we recover  the classical condition $\rho < 1$.

In the following example, we present a family of deterministic weights, such that the randomness of $P^Q$ only depends on the random connections of the agents to objects. 
		
\bexam[Deterministic weights] \\\label{ex-mbinom}
Let $W^i_j=\frac{r}{c_j}$ independent of $i$ with $r> 0$ independent of $i$ and $j$, such that \eqref{bedweights} holds. Then
\begin{align*}
P^Q 
&= 	\sum_{j=1}^d \sum_{i \in Q} \bone \{i \sim j\} 
\frac{\rho_j}{\sum_{i \in Q} \bone \{i \sim j\} +\sum_{k\neq j}  \sum_{i \in Q} \bone \{i \sim k\}  }.
		\end{align*}
		For $Q=\{i\}$ in the Bernoulli network with deterministic parameter $p\in(0,1]$, 
applying Theorem~1 of \cite{CS} on
$X^i=\sum_{k\neq j}  \bone \{i \sim k\} $ 
we find 
\begin{align*}\E \left[ \frac1{1+X^i}\right] &= \int_0^1 \E[u^{X_i}]  du= \int_0^1 \prod_{k\neq j} \E[u^{\bone\{i\sim k\}}]du  = \int_0^1 \prod_{k\neq j} (p u + (1-p)) du\\
& = \int_0^1  (p u + (1-p))^{d-1} du = \frac{1}{pd} (1-(1-p)^{d})
 \end{align*}
 and hence 
\begin{align*}
\E [P^i] &=  \sum_{j=1}^d \E\left[ \bone \{i \sim j\} \rho_j \frac1{1+X^i}\right]
= p \sum_{j=1}^d  \rho_j  \E \left[ \frac1{1+X^i}\right]  
= (1-( 1-p)^d)\frac{1 }{d}  \sum_{j=1}^d  \rho_j. 
\end{align*}
Thus the mean ruin parameter is independent of $i$ and equals the arithmetic mean of the ruin parameters of all objects multiplied by $\P(\deg(i)\neq 0)=1-(1-p)^d$.\\
Similarly, for $Q=\{i\}$ in a mixed Binomial network with random parameter $\Theta$, 
$$\E [P^i] 
= \frac{1 }{d}  \sum_{j=1}^d  \rho_j  \E \left(1-( 1-\Theta)^d \right).$$ 
In particular if $\Theta$ follows a Beta$(\alpha,\beta)$ distribution with density
$B(\alpha,\beta)^{-1} x^{\alpha-1} (1-x^{\beta-1}) \bone_{x\in[0,1]}$, 
\begin{align*}
\E[P^i] 
&=   \Big( 1- \frac{B(\alpha, \beta+d)}{B(\alpha,\beta)} \Big) \frac{1 }{d} \sum_{j=1}^d  \rho_j = \Big( 1- \frac{\beta (\beta+1) \cdots (\beta+d)}{(\alpha+\beta)(\alpha+\beta+1) \cdots (\alpha+\beta+d)}\Big) \frac{1}{d} \sum_{j=1}^d  \rho_j
\end{align*}
and again the expected ruin parameter is independent of $i$ and proportional to the arithmetic mean of the ruin parameters of all objects.\\
In the special case $\alpha=\beta=1$ where $\Theta$ follows a uniform distribution on $[0,1]$ the above simplifies to
$$\E[P^i]=\Big( 1- \frac{(d+1)!}{(d+2)!}\Big) \frac{1}{d} \sum_{j=1}^d  \rho_j=\frac{d+1}{(d+2)} \frac1{d} \sum_{j=1}^d  \rho_j,$$
such that  for $d\to\infty$ the expected ruin parameter converges from below to the arithmetic mean of the ruin parameters of all objects.
  \eexam
 
\subsection{$P^Q$  and $\Psi^Q$ in the complete network}\label{s42complete}

The complete network is particularly easy to treat. Here the network ruin parameter $P^Q$ from \eqref{networkrho} equals 
$$P^Q= \frac{\sum_{j=1}^d (\sum_{i \in Q} W^i_j)  \lambda_j \mu_j}{ \sum_{j=1}^d   (\sum_{i \in Q} W^i_j) c_j}, \quad Q\subseteq \{1,\dots,q\}.$$
In particular, if $\sum_{i \in Q} W^i_j=W^Q$ does not depend on $j$, then $P^Q= (\sum_{j=1}^d \lambda_j \mu_j)/( \sum_{j=1}^d c_j)$  is deterministic and does not depend on the choice of the set $Q$. 
This holds true in particular for homogeneous weights \eqref{homo}, where every object is equally shared by all agents that connect to it such that $\deg(j)=q$ for $j=1,\ldots, d$ and thus $\sum_{i \in Q} W^i_j=\frac{|Q|}{q}$.

For a fixed set $Q\subseteq\{1,\ldots, q\}$ of agents and an exponential system with proportional weights as in Theorem \ref{expjumpQ}, a complete network implies 
$$P^Q=  \frac{d \lambda}{\sum_{j=1}^d  c_j/\mu_j } =  \Big( \frac{1}{d} \sum_{j=1}^d \frac{1}{\rho_j}\Big)^{-1}   ,\quad  i=1,\ldots, q, $$
which again is deterministic. 
If $P^Q<1$, then in the exponential system we find from \eqref{ruinspecialQ}
$$\Psi^Q(u) = P^Q e^{-\frac{1-P^Q}{r^Q} \sum_{i \in Q} u^i} $$
for $u \in [0,\infty)^q $ such that $\sum_{i \in Q} u_i > 0$, which is similar to the one-dimensional case \eqref{exporuin}.

\bexam To illustrate the effect of the random network and the weights further, assume that the underlying bipartite network is itself a mixture - a complete graph with probability $\alpha \in (0,1)$, and the mixed Binomial model considered in Example \ref{ex-mbinom} with probability $1 - \alpha$. For a single agent $i\in Q$ we then have
$$
\E[P^i] =  \frac{1}{d} \sum_{j=1}^d \rho_j \left(\alpha + (1-\alpha)  \E \big[1-( 1-\Theta)^d \big]\right)
\le  \frac{1}{d} \sum_{j=1}^d \rho_j .
$$
The expected ruin parameter will be smaller than the ruin parameter of the deterministic network unless  the random network is a complete graph almost surely. 

In general this monotone behaviour is not the case. For example, in an exponential system with proportional weights  as in Theorem \ref{expjumpQ}, the same graph mixture gives for a single agent $i$ that 
$$
\E[P^i] =  \alpha d \, \Big(  \sum_{j=1}^d \frac{1}{\rho_j}\Big)^{-1} + (1 - \alpha) \E 
\Big[  \frac{{\rm{deg}}(i)} {\sum_{j=1}^d \frac{1}{\rho_j} \bone ( i \sim j)} \Big].
$$
Which one of the two summands dominates the expectation depends on the network model and on $\alpha$. 
\eexam

\section{The joint ruin probability of a set of agents}\label{s4}

In this section we consider $\Psi_{\wedge}^Q$ as defined in \eqref{jointhit2}. Due to the far more complicated structure of the process $\min_{i\in Q} (R^i(t)-u^i)$ compared to the sum of components, we do not obtain an explicit form for $\Psi^Q_\wedge$. 
Still, we can derive a Lundberg-type bound for $\Psi^Q_\wedge$ using classical martingale techniques. 

Recall the bound $W^i$ in \eqref{wW} and write for any $Q\subseteq \{1,\ldots,q\}$ and any two vectors $a,b \in \R^q$
$$\langle a,b\rangle_Q:=\sum_{i\in Q} a_i b_i.$$

\begin{theorem}\label{thm:multlundberg}  {\em [Network Lundberg bound for joint ruin probabilities of several agents]} \\
	Let $Q\subseteq\{1,\ldots, q\}$ be a set of agents and assume that for all $j\in \{1,\ldots,d\}$ the cumulant generating functions $\varphi_{j}(t):=\log \E e^{t V_j(1)}$ exist in some neighbourhood of zero. Then for fixed $\bfa\in\A$,
\begin{equation}\label{lundbergrealmin}
\P\Big(\min_{i\in Q}(R^i(t)-u^i)\geq 0 \;\text{for some }t\geq 0 \mid A=\bfa\Big)
\leq \mathds{1}\{\min_{i\in Q} \deg(i)>0\} e^{-\langle \kappa_\wedge(\bfa,u), u\rangle_Q},
\end{equation}
for $u\in[0,\infty)^q, \sum_{i \in Q}u^i\neq 0,$ where 
$$\kappa_\wedge(\bfa,u)=\underset{\substack{r\in (0,\infty)^q:\\ \varphi_j(\sum_{i\in Q} r^i a^i_j)\leq 0}}{\arg\max} \langle r, u\rangle_Q .$$
	In particular, assume that for all objects $j=\{1,\dots,d\}$ the adjustment coefficient $\kappa_j\in(0,\infty)$ satisfying \eqref{lucoeff} exists. Then
	$$\Psi^Q_\wedge(u) 	\leq \P\big(\min_{i\in Q} \deg(i)>0\big) e^{-\langle \kappa_\wedge(u) ,u\rangle_Q}, \quad u\in[0,\infty)^q, \sum_{i \in Q}u^i\neq 0,$$
	with 
	$$\kappa_\wedge(u)= \underset{\substack{r\in (0,\infty)^q:\\ \sum_{i\in Q}  
		r^iW^i \leq \min \{\kappa_j, j=1,\ldots, d\}}}{\arg\max} \langle r,u \rangle_Q.$$
\end{theorem}

\begin{remarks}\label{compare}\rm [Comparing the bounds for $\Psi^Q(u)$  and $\Psi^Q_\wedge(u)$]\\
(i) Assume that all objects $j\in\{1,\dots,d\}$ have the same adjustment coefficient $\kappa$. Let $Q\subseteq \{1,\dots,q\}$ be a set of agents and assume for the risk reserve $u^i=U/|Q|$ for $i\in Q$ and $U=\sum_{i \in Q}u^i>0$. Then $\kappa_\wedge(u)=\kappa/\sum_{i\in Q} W^i \cdot (1,\ldots,1)$, which gives the bounds
\beao
\Psi^Q_\wedge(u) & \leq & \P(\min_{i\in Q} \deg(i)> 0 ) e^{-\kappa U /\sum_{i\in Q} W^i}, \quad \text{and}\\
\Psi^Q(u) & \le & \P(\deg(Q)> 0 ) e^{-\kappa U/\sum_{i\in Q} W^i},
\eeao
from Theorem \ref{prop:lundbergbound}.
The exponential decay is for both ruin probability bounds the same. The constant in $\Psi^Q_\wedge(u)$ is, however, in general smaller than in $\Psi^Q(u)$.\\
(ii) For $Q=\{i\}$ we have $\Psi^{i}=\Psi_\wedge^{i}$ and also the bounds obtained in Theorems \ref{prop:lundbergbound} and \ref{thm:multlundberg} coincide.
\end{remarks}

%%%%%%%%%%%%%%%%%%%%%%%%%%%%%%%%%%%%%%%%%%%%%%%%%%%%%%%%%%%%%%%%%%%%%%%%%%%
\section{Proofs}\label{proof}
%%%%%%%%%%%%%%%%%%%%%%%%%%%%%%%%%%%%%%%%%%%%%%%%%%%%%%%%%%%%%%%%%%%%%%%%%%%%

Throughout, we shall denote all realisations of random quantities which are influenced by the realisations $\bfa\in \mathbb{A}$ of the network structure, by the corresponding tilded letters; e.g., $\tilde R=\tilde{R}(\bfa)$ is a specific realisation of  the process $R$ when the network $\bfa$ is fixed. 

\subsubsection*{Proof of Theorem~\ref{prop_psisum}}

	By definition of the process $(R(t))_{t\ge 0}$ we have 
	\begin{align*}
	\sum_{i\in Q} (R^i(t)-u^i) & = \sum_{j=1}^d   \Big(\sum_ {i\in Q} A^i_j \Big) V_j(t)  - \sum_ {i\in Q} u^i,\quad t\ge 0,
	\end{align*}
	such that 
	$$\Psi (u) = \P\Big( \sum_{j=1}^d   \Big(\sum_ {i\in Q} A^i_j\Big) V_j(t) \geq \sum_ {i\in Q} u^i  \text{ for some } t\geq 0\Big),\quad u\in [0,\infty)^q, \sum_{i \in Q}u^i > 0.$$
	For every realisation $\bfa=(a^i_j)$  of the network with $\deg(Q)> 0$ the process
	$( \sum_{j=1}^d ( \sum_ {i\in Q}  a^i_j) V_j(t) )_{t\ge0}$
	is a compound Poisson process with intensity, claim size distribution and drift given by 
	\begin{align*}
	\tilde{\lambda}= \sum_{j=1}^d \bone\{ Q\sim j\} \lambda_j,\quad \quad
	\tilde{F}(x)&= \frac{1}{\tilde{\lambda}} \sum_{j=1}^d \lambda_j \bone\{ Q\sim j\} F_j\Big(\frac{x}{\sum_ {i\in Q}  a^i_j}\Big),  \quad \text{and} \quad
	\tilde{c}=	\sum_{j=1}^d \Big( \sum_ {i\in Q} a^i_j\Big) c_j.
	\end{align*}
	Hence, whenever 
	$$\tilde{\rho}:=\frac{\sum_{j=1}^d (\sum_ {i\in Q}  a^i_j) \lambda_j \mu_j}{\sum_{j=1}^d ( \sum_ {i\in Q} a^i_j) c_j} <1,$$
	for any fixed realisation $\bfa=(a^i_j)$ of $A$ it holds that
	\begin{align*}
	\P\Big( \sum_{j=1}^d   \Big(\sum_ {i\in Q} a^i_j\Big) V_j(t) \geq \sum_ {i\in Q} u^i  \text{ for some } t\geq 0\Big)
	&= (1- \tilde\rho) \sum_{n=1}^\infty \tilde\rho^{n} \ov{(\tilde{F})_I^{n\ast}}(\sum_ {i\in Q} u^i),
	\end{align*}
	by the classical Pollaczek-Khintchine formula \eqref{PK}. For $\deg(Q)=0$ the ruin probability is obviously $0$. The result now follows by conditioning on the realisations of $A$ since $A$ and $V$ are independent.

\subsubsection*{Proof of Theorem \ref{prop:lundbergbound}}

	Our proof relies on standard martingale arguments which is why we will only briefly sketch it here.
	Note first that for any realisation $\bfa \in \A$ with $\deg(Q)=0$ obviously ruin cannot occur. Thus fix $\bfa\in \A$ such that $\deg(Q)\neq 0$. Then the mgf of $\sum_{i\in Q} \tilde{R}^i(t)$ can be computed as
	\begin{align*}
	\E\Big[\exp\Big(r \sum_{i\in Q} \tilde{R}^i(t)\Big)\Big]  
	&=\E\Big[\exp\Big(r\sum_{i\in Q} \Big(\sum_{j=1}^d a^i_j V_j(t)\Big)\Big)\Big] = \prod_{j=1}^d \E\Big[r \Big(\sum_{i\in Q}  a^i_j\Big) V_j(t) \Big] 
	=\exp\Big(t \sum_{j=1}^d \varphi_j\Big(r\sum_{i\in Q} a^i_j\Big) \Big) \nonumber\\
	&=: \exp( t g_\bfa(r)), 
	\end{align*}
	for any $r\geq 0$ such that the occuring terms are finite. This yields by standard arguments that for all $\mathbf{u}:=\sum_{i\in Q} u^i>0$
	$$M_{\bfa,u}(t):= \frac{\exp\Big(r\Big(\sum_{i\in Q} \tilde{R}^i(t) -\mathbf{u}\Big)\Big)}{\exp(tg_\bfa(r))},\quad t\ge 0,$$ 
	is a martingale with respect to the natural filtration of $(V(t))_{t\ge0}$. Proceeding as in the classical proof of the Lundberg bound (see e.g. Proposition~3.1 of \cite{AsAl}) we obtain
	$$\P\Big(\sum_{i\in Q} (R^i(t)-u^i)\geq 0 \; \text{for some } t\geq 0 \mid A=\bfa\Big) \leq e^{-r\mathbf{u}} \sup_{t\geq 0} \exp(tg_{\bfa}(r)),$$
	which yields \eqref{lundbergreal} with
	$\kappa(\bfa)=\sup\{r>0: g_\bfa(r)\leq 0\}.$\\
	To obtain the global bound we have to choose $\kappa$ as large as possible such that 
	$g_\bfa(\kappa)\leq 0$ for all $\bfa\in \A$. This is clearly satisfied if $\varphi_j(\kappa \sum_{i\in Q} a^i_j)\leq 0$ for all $j$ and all $\bfa\in \A$ which leads to the form given in \eqref{kappaglobal}. Thus
	$$\P\Big(\sum_{i\in Q} (R^i(t)-u^i)\geq 0 \; \text{for some } t\geq 0\mid A=\bfa\Big) \leq e^{-\kappa \mathbf{u}},$$
	for all $\bfa\in \A$, and with 
	$$\Psi(u)=\int_\A \P\Big(\sum_{i\in Q} (R^i(t)-u^i)\geq 0 \; \text{for some } t\geq 0\mid A=\bfa\Big) d\P_A(\bfa)$$
	we obtain the result.

\subsubsection*{Proof of  Theorem~\ref{expjumpQ}}

We calculate the random integrated tail distribution $F_I^Q$ as in  \eqref{FiQ},  
\begin{align*}
F^Q_I(x)
&= \Big(\sum_{j=1}^d \Big(\sum_{i\in Q}  \bone\{i \sim j\} W_j^Q\Big)\mu_j  \Big)^{-1} \sum_{j=1}^d \bone \{Q\sim j\}  \int_0^x \overline{F_j} \Big(\frac{y}{\sum_{i \in Q} \bone\{i \sim j\} W_j^Q}\Big) dy\\
&= \Big(\sum_{j=1}^d \bone\{Q\sim j\} r^Q\Big)^{-1} \sum_{j=1}^d \bone \{Q\sim j\}   \int_0^x \overline{F_j} \Big(\frac{y\mu_j}{ \bone\{Q\sim j\}r^Q}\Big) dy\\
&= \Big(\sum_{j=1}^d \bone \{Q \sim j\}\Big)^{-1} \sum_{j=1}^d \bone \{Q\sim j\}   \big(1-e^{-x/r^Q}\big)\\
&= 1-e^{-x/r^Q}, \quad x\ge 0,
\end{align*}
which is  deterministic; we recognise it as the distribution function of the  exponential distribution  with mean $r^Q$. 
Hence $(F^Q_I)^{n\ast}$ is an Erlang distribution function with density 
$$g^Q_n(x) = \frac{x^{n-1}}{(n-1)!(r^Q)^n}  e^{-x/r^Q},\quad x\ge 0.$$  
Moreover, due to the assumptions on the network, $P^Q$ in  \eqref{networkrho} equals
\eqref{networkrhoerlangQ}.
From \eqref{ruinsumpartial} we obtain for $u\in [0,\infty)^q$ such that $\sum_{i\in Q} u^i > 0$,
\begin{equation} \label{intermediatetailQ}  
\Psi^Q (u) 
= \P(P^Q\ge 1) +  \P(P^Q<1) 
 \E \Big[(1- P^Q) 
  \sum_{n=1}^\infty (P^Q)^{n}  \frac{1}{(n-1)!(r^Q)^n } 
\int_{\sum_{i\in Q} u^i}^\infty t^{n-1}e^{-t/r^Q} dt   \Mid   P^Q<1 \Big].
\end{equation}
Now we calculate that 
\begin{align*}
\sum_{n=1}^\infty  \left(\frac{P^Q}{r^Q}\right)^{n}    \frac{ t^{n-1}}{(n-1)! } 
&= \frac{P^Q}{r^Q}   e^{  P^Q t/r^Q},\quad t\ge 0,
\end{align*}
and 
\begin{align*}
\int_{\sum_{i\in Q} u^i}^\infty  \sum_{n=1}^\infty \left(\frac{P^Q}{r^Q}\right)^{n}
\frac{t^{n-1}}{(n-1)!}  e^{-t/r^Q} dt  
& =  \frac{P^Q}{r^Q}\int_{\sum_{i\in Q} u^i}^\infty  e^{ -t( 1 -  P^Q)/r^Q }  dt
= \frac{P^Q}{ 1 - P^Q}  e^{- \sum_{i\in Q} u^i ( 1 -  P^Q)/r^Q}. 
\end{align*}
Using this expression in \eqref{intermediatetailQ} gives the assertion. 

\subsubsection*{Proof of Proposition \ref{Poihomo}}

	We write 
	\begin{equation*}
	\deg(j) = \bone\{ i \sim j\} + \deg^{(i)}(j), \quad \text{and} \quad \deg_Q(j) = \sum_{i\in Q} \bone \{i\sim j\} = \bone \{i\sim j\}+ \deg_Q^{(i)}(j),
	\end{equation*}
	so that $\deg^{(i)}(j) $ and $\bone\{ i \sim j\} $ are independent, as well as $\deg^{(i)}_Q(j) $ and $\bone\{ i \sim j\} $. 
	Recall from Remark \ref{deg0notnecessary} that in 
	\eqref{networkrho} the indicator  $\bone \{ \deg(Q) >  0\}$ can be omitted. 
	Thus, 
	\begin{align*}
	P^Q 
	&=  \frac{ \sum_{j=1}^d \sum_{i \in Q}\bone \{i \sim j\} \lambda_j \mu_j / \deg(j)   }{\sum_{\ell=1}^d \sum_{i \in Q}\bone \{i \sim \ell\} c_\ell / \deg(\ell)  }\nonumber \\
	&=  \sum_{j=1}^d \sum_{i \in Q}\bone \{i \sim j\}  \frac{\lambda_j \mu_j}{ \deg(j) \sum_{\ell=1}^d \sum_{s \in Q}\bone \{s \sim \ell\} c_\ell/\deg(\ell)  }
	\nonumber \\
	&=  \sum_{j=1}^d \sum_{i \in Q}\bone \{i \sim j\}  \frac{\lambda_j \mu_j}{ c_j (1+ \deg^{(i)}_Q(j)) +  \sum_{\ell \ne j} \sum_{s\in Q}\bone \{s \sim \ell\} c_\ell  ( 1 + \deg^{(i)}(j)) /( 1 + \deg^{(s)}(\ell)) }.  
	\end{align*} 
	 Note that for objects $j$ and $\ell$, $\deg^{(i)}(j)$ and $\deg^{(s)}(\ell)$ are independent for $\ell \ne j \in \{1,\ldots, d\}$ and $i\neq  s \in \{ 1, \ldots, q\}$.  
Thus $P^Q$ is expressed as a function of the edge indicators with the dependence disentangled. While $P^Q$ is a non-negative function of the edge indicators, it does not quite fit into the framework of Proposition \ref{poissoncor} because it may not be bounded by 1. Instead, the deterministic expression $\sum_{j=1}^d \rho_j$ serves as upper bound. Using the function 
$ g( \{  \bone \{i \sim j\},  i=1, \ldots, q,  j=1, \ldots, d \}) = h \left(  (\sum_{j=1}^d \rho_j)^{-1} P^Q \right) $ 
with $h \in [0,1]$ makes  Proposition \ref{poissoncor} applicable. Equivalently, instead of transforming $g$  the bound \eqref{poissonbound} can be multiplied by $\sum_{j=1}^d \rho_j$.

	We apply Proposition~\ref{poissoncor} and use that the sum of independent Poisson variables 
	 is again Poisson, so that we approximate 
	$\deg^{(i)}_Q (j)  $ by $ Z^{(i)}_j$, and $ \deg^{(i)} (j)  $ by $ Z^{(i)}_j + \bar{Z}^{(i)}_j. $
	Moreover,  the Poisson variables $Z^{(i)}_{j}$ and $Z^{(i)}_{\ell}$ are independent for $\ell \ne j \in \{1,\ldots, d\}$ and independent of $\bone\{i\sim j\}$. The same is true for $\bar{Z}^{(i)}_{j}$,  $\bar{Z}^{(i)}_{\ell}$ and $\{Z_{i,j}, i=1, \ldots, q, j=1, \ldots, d\}$. 
The assertion now follows from Proposition \ref{poissoncor}.

\subsubsection*{Proof of Proposition \ref{Poiexpo}}

{For} proportional weights as in \eqref{AlaQ},   using \eqref{networkrho2} for \eqref{networkrhoerlangQ} we have, regardless of the claim size size distribution, 
	\begin{align*}
	P^Q = & \sum_{j=1}^d\bone (Q \sim j) \frac{ \lambda
	}{\frac{c_j}{\mu_j} + \sum_{\ell \ne j} \bone (Q \sim \ell) \frac{c_\ell}{\mu_\ell} }
	= \sum_{j=1}^d\bone (Q \sim j) \frac{1
	}{\frac{1}{\rho_j } + \sum_{\ell \ne j} \bone (Q \sim \ell) \frac{1}{\rho_\ell } }.
	\end{align*}
	{Thus $P^Q$ is a non-negative function of the edge indicators, and it can be bounded above by $\sum_{j=1} \rho_j$.} 
	Now for $j =1, \ldots, d$ {and $\theta \in [0,1]$} let
	$$\pi_{Q,j} {(\theta)}  = \P (Q \sim j {\mid \Theta = \theta }) = 1 - \prod_{i \in Q} ( 1 - p_{i,j}{(\theta)}).$$
	Then the Poisson approximation of Proposition~\ref{poissoncor} gives {the assertion}.

\subsubsection*{Proof of Theorem \ref{thm:multlundberg}}

For notational simplicity the following proof is only given for $Q=\{1,\ldots, q\}$. The general case can easily be obtained by cutting down the network to a subset of agents.
As in the proof of Theorem \ref{prop:lundbergbound} we will follow a standard martingale approach.
Note first that, if $\deg(i)=0$ for one or more agents, then the joint ruin probability is zero, since at least one component of the process $R$ is constant. Thus fix any realisation $\bfa$ of $A$ such that $\deg(i)> 0$ for all $i$.
The mgf of $\bfa V(t)$ can be computed as
 \begin{align*}
\E\Big[\exp\big(\langle r, \bfa V(t)\rangle\big)\Big]
&=\E\Big[\exp\Big(\sum_{i=1}^q r_i \Big(\sum_{j=1}^d a^i_j V_j(t)\Big)\Big)\Big] = \prod_{j=1}^d \E\Big[\sum_{i=1}^q r_i a^i_j V_j(t) \Big] 
%= \prod_{j=1}^d \exp(t \psi_{V_j}(\sum_{i=1}^q r_i a^i_j))
=\exp\Big(t \sum_{j=1}^d \varphi_{j}\Big(\sum_{i=1}^q r_i a^i_j\Big) \Big) \nonumber\\
&=: \exp( t h_\bfa(r)), 
\end{align*}
for any $r=(r_1,\ldots, r_q)\in(0,\infty)^q$ such that the occuring terms are finite. Hence, for these $r$
$$M_{\bfa, u}(t)=\frac{\exp(\langle r, \bfa V(t)-u\rangle )}{\exp(th_\bfa (r))}, \quad t\geq 0,$$
is a martingale with respect to the natural filtration of $(V(t))_{t\ge0}$. 
Proceeding via Doob's optional stopping theorem as in the classical proof of the (one-dimensional) Lundberg bound (see e.g. Proposition~3.1 of \cite{AsAl}) we obtain for any $r\in(0,\infty)^q$
\begin{equation*}
\P\big(\bfa V(t)-u\in [0,\infty)^q \;\text{for some }t\geq 0\big)
\leq e^{-\langle r, u\rangle} \sup_{t\geq 0} e^{t h_\bfa(r)},
\end{equation*}
which proves \eqref{lundbergrealmin}. \\
For the global bound note that for any $r$ such that  $\sum_{i=1}^q r^i W^i \leq \kappa_j$ we have 
$\sum_{i=1}^q r^i a^i_j \leq \kappa_j$ for all $j$ and hence
 $\varphi_j(\sum_{i=1}^q r^i W^i )\leq \varphi_j(\kappa_j)=0$, $j=1,\ldots, d$. Thus if $\sum_{i=1}^q r^i W^i \leq \min \{\kappa_j, j=1,\ldots,d \}$ this yields 
$h_\bfa(r) \leq 0$ and 
$$\P\big(\bfa V(t)-u\in [0,\infty)^q \;\text{for some }t\geq 0\big) \leq e^{-\langle r, u\rangle}$$
for any realisation $\bfa$, which gives the result.

\begin{small}
\bibliography{bibgesine}
\bibliographystyle{plain}
\end{small}

\end{document}